\documentclass[11pt,reqno]{amsart}
\usepackage{amsmath,amsfonts,epsfig,amssymb,graphics,enumerate}
   \usepackage{hyperref,color}

\newtheorem{thm}{Theorem}[section]

\newtheorem{cor}{Corollary}[section]

\newtheorem{lemma}{Lemma}[section]
\newtheorem{defn}{Definition}[section]
\newtheorem{rem}{Remark}[section]
\newcommand{\C}{\mathbb{C}}
\newcommand{\N}{\mathbb{N}}
\newcommand{\R}{\mathbb{R}}
\newcommand{\ep}{\varepsilon}
\newcommand{\T}{\mathcal{T}}
\begin{document}

\title[Eigenvalues of the Laplacian and extrinsic geometry ]{Eigenvalues of the Laplacian and extrinsic geometry}
\author{Asma Hassannezhad}
\address{Universit\'e de Neuch\^atel, Institut de Math\'ematiques, Rue Emile-Argand 11, Case postale 158, 2009 Neuch\^atel
Switzerland}
\thanks{The  author has benefitted from the support of boursier du gouvernement Fran\c cais during her stay in Tours.}

\date{}
\begin{abstract}
We extend the results given by Colbois, Dryden and El Soufi on the relationships between the eigenvalues of the Laplacian and an extrinsic invariant called \textit{intersection index}, in two directions. First, we  replace this intersection index by invariants of the same nature which are stable under
\textit{small perturbations}. Second, we consider complex submanifolds of the complex projective space $\C P^N$ instead of submanifolds of $\R^N$ and we obtain an eigenvalue upper bound depending only on the dimension of the submanifold which is sharp for the first non-zero eigenvalue.
\end{abstract}




\maketitle

\section{Introduction and statement of the results}
For a compact manifold without boundary, the spectrum of the Laplace-Beltrami operator $\Delta$ consists of an unbounded nondecreasing sequence of nonnegative real numbers
\[0=\lambda_1<\lambda_2\leq\cdots\leq\lambda_k\leq\cdots\nearrow\infty,\]
where each eigenvalue $\lambda_k$ has finite multiplicity.
The study of the relationship between the extrinsic
geometry of  submanifolds and the spectrum of the Laplace-Beltrami
operator is {an} important topic of  spectral geometry. 
One of the well-known extrinsic invariants is the mean curvature vector field of a submanifold. In this regard, we can mention the Reilly inequality \cite{Re} for an immersed $m$-dimensional submanifold $M$ of $\R^{N}$
 $$ \lambda_2(M)\le\frac m {\text{Vol}(M)} \|H(M)\|_2^2,$$
 where $\|H(M)\|_2$ is the $L^2$-norm of the mean curvature vector field of $M$. For higher eigenvalues,  it follows from results of El Soufi, Harrell and Ilias  \cite{EHI} that for every $k\in\N^*$,
  $$ \lambda_k(M)\le R(m) \|H(M)\|_\infty ^2 \ k^{2/m},$$
where $\|H(M)\|_\infty$ is the $L^\infty$-norm of $H(M)$ and $R(m)$ is a constant depending only on $m$. Since, the variational characterization of eigenvalues do not depend on  derivatives of the metric, we are interested in extrinsic invariants which {do} not depend on metric derivatives, excluding for instance  curvature. The \textit{intersection index} (see below for the definition) is an important example of such intrinsic invariants. In \cite{CDE}, Colbois, Dryden and El Soufi studied
 the relation between   the intersection index, and {the}
eigenvalues of the Laplace-Beltrami operator.
 In this paper, we review and extend their results.\\ 
  For a compact
$m$-dimensional immersed submanifold $M$ of $\R^N=\R^{m+p}$,
$p>0$, the \textit{intersection index} is given by
\[\label{indexch3p}
i(M)=\sup_\Pi\sharp (M\cap\Pi),
\]
where $\Pi$ runs over the set of all $p$-planes that are
transverse to $M$; if $M$ is not embedded, we count multiple
points of $M$ according to their multiplicity.
We remark that the intersection  index was also investigated by Thom  \cite{ChM} where it was called the \textit{degree} of $M$.
In \cite{CDE}, Colbois, Dryden and El Soufi show that there is a
positive constant $c(m)$, depending only on  $m$, such that for
every compact $m$-dimensional immersed submanifold $M$ of $\R^{m+p}$,
we have the following inequality
\begin{equation}\label{cde}
 \lambda_k(M)\textrm{Vol}(M)^{2/m}\leq c(m)i(M)^{2/m} k^{2/m}.
\end{equation}
Moreover, the intersection index in the above inequality  is not
replaceable with a constant depending only on the dimension $m$.
Even for hypersurfaces, the first positive eigenvalue cannot be
controlled only in terms of the volume and the dimension (see
\cite[Theorem 1.4]{CDE}).
As an immediate consequence of Inequality (\ref{cde}), the normalized eigenvalues on convex hypersurfaces are
bounded above only in terms of the dimension. Another remarkable consequence of Inequality (\ref{cde}) concerns algebraic submanifolds \cite[Corollary
4.1]{CDE}: Let $M$ be a compact real algebraic manifold, i.e. $M$
is a zero locus of $p$ real polynomials in $m+p$ variables of
degrees $N_1,\dots,N_p$. Then
\begin{equation}\label{poly}
\lambda_k(M)\textrm{Vol}(M)^{2/m} \leq   c(m) N_1^{2/m}\cdots
N_p^{2/m} k^{2/m}.
\end{equation}
Note that Inequalities (\ref{cde}) and \eqref{poly} are not stable under
``\textit{small}'' perturbations, since the intersection index
might dramatically change. \

 We extend the work of  Colbois, Dryden and El Soufi in two directions. The first one consists in replacing the intersection index  $i(M)$ by  invariants of  the same nature which are stable under
\textit{small perturbations}. The second direction concerns complex submanifolds of the complex projective space $\C P^N$. Here we obtain an eigenvalue upper bound for submanifolds of $\C P^N$ depending only on the dimension. Below we describe the main results of this paper. \\

{\bf First part.}
Let
$\varepsilon<1$ be a positive number.  By  an \textit{$\varepsilon$-small
perturbation}, we mean any perturbation in a region $D\subset M$
whose measure is at most equal to  $\varepsilon\rm{Vol}(M)$. To
avoid any technical complexity, we
 assume that $M\setminus D$ is a smooth manifold with  smooth boundary. Here, we  define new notions of intersection indices which are stable under any $\varepsilon$-small perturbation.
Let $G$ be the Grassmannian  of all
 $m$-vector spaces in $\R^{m+p}$ endowed with the $O(m+p)$- invariant
Haar measure of total measure 1. Let $0<\varepsilon<1$ and $D$
be any open subdomain of $M$ such that $M\setminus D$ is a smooth
manifold with  smooth boundary and
$\textrm{Vol}(D)\leq\varepsilon\textrm{Vol}(M)$. We denote
$M\setminus D$  by $M^D_\varepsilon$. Let $H$ be an $m$-vector space in  $G$. We  define $i_H(M^D_\varepsilon):=\sup_{P\bot H}\sharp
(M^D_\varepsilon\cap P)$, where $P$ runs over
 affine $p$-planes orthogonal to $H$. We now define the
\label{imathdef}\textit{$\varepsilon$-mean intersection index}  as follows:
\[
 \bar{\imath}^\varepsilon(M):=\inf_D\int_G i_H(M^D_\varepsilon)
 dH,
\]
{where $D$ runs over regions whose measure is smaller than $\varepsilon\textrm{Vol}(M)$ and $M\setminus D$ is a smooth
manifold with  smooth boundary. }\\
 Similarly, for $r>0$, we define the \textit{
$(r,\varepsilon)$-local intersection index} as:
\[
 \bar{\imath}^\ep_r(M):=\inf_D\sup_{x\in M^D_\varepsilon}\int_G i_H(M^D_\varepsilon\cap B(x,r)) dH,
\]
where $B(x,r)\subset \R^{m+p}$ is an Euclidean  ball centered at $x$ and of radius $r$ and {$D$ runs over regions whose measure is smaller than $\varepsilon\textrm{Vol}(M)$ and $M\setminus D$ is a smooth
manifold with  smooth boundary. }\\
We can now state our theorem.

\begin{thm}\label{a3}
  There exist positive constants $c_m, \alpha_m$ and $\beta_m$ depending only on $m$ such that for every compact $m$-dimensional immersed submanifold $M$ of $\R^{m+p}$, every $r>0$,
 $k\in \N^*$, and $0<\ep<1$, we have
\begin{equation}\label{a2}
 \lambda_k(M)\textrm{Vol}(M)^{2/m}\leq   c_m \frac{\bar{\imath}^\varepsilon(M)^{2/m}}{{(1-\varepsilon)}^{1+2/m}} k^{2/m},
\end{equation}
and
\begin{equation}\label{a1}
 \lambda_k(M)\leq\alpha_m\frac{1}{(1-\ep)r^2}+\beta_m\frac{\bar{\imath}^\varepsilon_r(M)^{2/m}}{{(1-\varepsilon)}^{1+2/m}}\left(\frac{k}{\textrm{Vol}(M)}
\right)^{2/m}.
\end{equation}
\end{thm}
The main feature of the inequalities \eqref{a2} and \eqref{a1} is
that the upper bounds are not considerably affected by the
presence of a large intersection index in a ``small'' part of $M$
(i.e. a subdomain with small volume). In particular, for a
compact hypersurface of $\R^{m+1}$ which is convex outside a
region\footnote{We say that $M$ is convex outside of $D$ if after a perturbation of  $M$ which is the identity outside of  $D$ we get a convex compact hypersurface.}  $D$ of measure at most $\varepsilon\rm{Vol}(M)$, one has
$\bar{\imath}^\varepsilon(M)\leq i(M^D_\varepsilon)$ and then
\begin{equation*}
 \lambda_k(M)\textrm{Vol}(M)^{2/m}\leq   c_m \frac{2^{2/m}}{{(1-\varepsilon)}^{1+2/m}}
 k^{2/m}.
\end{equation*}
We also note that one has Inequality \eqref{poly}  not only for  compact algebraic submanifolds of  $\R^{N}$, but also for every $\varepsilon$-perturbation of those algebraic submanifolds, where the constant $c(m)$ in \eqref{poly} depends only on $m$ and on $\varepsilon$.\\

%
{\bf Second part.} We study
another natural context where algebraic submanifolds {can} be
considered which is the complex projective space $\C P^N$. According to
Chow's Theorem (\cite{GH}), every complex submanifold $M$ of $\C P^{N}$ is a
smooth algebraic variety, i.e. it is a zero locus of a family of
complex polynomials. We obtain the following upper bound for
complex submanifolds of $\C P^{N}$ endowed with Fubini-Study
metric $g_{_{FS}}$.
\begin{thm}\label{01}
Let $M^m$ be an $m$-dimensional  complex manifold admitting a
holomorphic immersion $\phi:M\to\C P^N$. Then for every
 $k\in\N^*$ we have
\begin{equation}\label{chap2cpn}\lambda_{k+1}(M,\phi^*g_{_{FS}})\leq
2(m+1)(m+2)k^{\frac{1}{m}}-2m(m+1).
\end{equation}
\end{thm}
In particular, one has Inequality (\ref{chap2cpn}) for every
complex submanifold of $\C P^N$. Note that the power of $k$ is
compatible with the
 Weyl law.\\
Under the assumption of Theorem \ref{01}, for $k=1$, one has
\begin{equation}\label{mogh1}\lambda_2(M,\phi^*g_{_{FS}})\leq4(m+1),
\end{equation}
which is a sharp inequality since the equality holds for $\C P^m$.
Inequality \eqref{mogh1} is obtained by Bourguignon, Li and
Yau \cite[page 200]{BLY},  and also by Arezzo, Ghigi
and Loi \cite{AGL}. Note that the results in \cite{AGL} and \cite{BLY} are for the first non-zero eigenvalue of Laplacian on a larger family of complex manifolds (see page \pageref{chp3full}). However,    Theorem \ref{01}  gives an upper bound for higher eigenvalues in addition to  giving
the sharp upper bound for $\lambda_2$, when we consider  the  complex submanifolds of $\C P^N$ endowed with the Fubini-Study metric. 
 For a complex submanifold
$M$ of $\C P^{m+p}$ of the complex dimension $m$, we have
\begin{equation}\label{ezafe1}\textrm{Vol(M)}=\deg(M)\textrm{Vol}(\C P^m),\end{equation} where $\deg(M)$ is the intersection number of $M$ with a projective $p$-plane in a generic
position  (see for example \cite[pages 171-172]{GH}). 
Multiplying  Inequality (\ref{chap2cpn}) by \eqref{ezafe1}, we get
\begin{equation}\label{al}
\lambda_{k+1}(M,g_{_{FS}})\textrm{Vol}(M)^{\frac{1}{m}}\leq
C(m)\deg(M)^{\frac{1}{m}}k^{\frac{1}{m}}.
\end{equation}Moreover, one can describe $M$ as a zero locus of a
family of irreducible homogenous polynomials and then $\deg(M)$
is bounded by the multiplication of degrees of the irreducible polynomials
 describing $M$.
One can now compare Inequality (\ref{al}) with Inequality (\ref{poly}).

This paper is organized as follows. In Section \ref{sec2ch2}, we recall one of the main methods to estimate the eigenvalues  in the abstract setting of metric measure spaces introduced by Colbois and Maerten \cite{CM} .     We use this method to prove Theorem
\ref{a3}
  in Section \ref{sec3ch2}. In Section \ref{sec4ch2}, we consider algebraic submanifolds of $\C
P^N$ and we prove Theorem \ref{01}. The method which is used in
Section \ref{sec4ch2} to show Theorem \ref{01} is independent from what we
introduce in Sections \ref{sec2ch2} and \ref{sec3ch2}.
\section*{Acknowledgement}
This paper is a part of the author's PhD thesis under the direction
of Professors Bruno Colbois (Neuch\^atel University), Ahmad El
Soufi (Fran\c{c}ois Rabelais University), and Alireza
Ranjbar-Motlagh (Sharif University of Technology) and she acknowledges their support and encouragement. The author
wishes also to express her thanks to Bruno Colbois and Ahmad El Soufi for suggesting the
problem and for many helpful discussions. She is also grateful to Mehrdad Shahshahani and the referee for helpful comments.
 \section{A general preliminary result}\label{sec2ch2}
A classical way  to estimate the eigenvalues of the Laplacian is to construct a family of disjoint domains and then, to estimate the Rayleigh quotients of the test functions supported on these domains. In \cite{CM}, Colbois and Maerten introduce a method to construct an elaborated family of disjoint domains in the general setting of metric-measure ($m-m$) spaces. This method shows that eigenvalue upper bounds and controlling the local volume concentration of  balls are linked.  Here, for an $m$ dimensional Riemannian submanifold $M$ of $\R^N$, controlling the local volume concentration of balls means to control the constant $C$ in the following inequality for some $\rho>0$
$${\rm Vol}(M\cap B(x,r))\leq Cr^m \quad \forall x\in M,~~0<r\leq \rho, $$
where $B(x,r)$ is a ball of radius $r$ centered at $x$ in $\R^N$.\\
This section is devoted to recall this construction for metric measure spaces.  Throughout this
section the triple $(X,d,\mu)$ will designate a complete locally
compact $m-m$ space with a {distance} $d$ and a finite, positive,
non-atomic Borel measure $\mu$. We also assume that balls in
 $(X,d)$ are
pre-compact. Each pair $(F,G)$ of Borel sets in $X$ such that $F\subset G$ is
called a capacitor. For $F\subseteq X$
and $r>0$, we denote the $r$-{\it neighborhood of} $F$  by $F^r$,
that is
\[  F^r=\{x\in X : d(x,F)\leq r\}.\]
\begin{defn}Given $\kappa>1$,  $\rho>0$  and $N\in\mathbb{N}^*$,
 we say that a metric space
$(X,d)$ satisfies the  $(\kappa,N;\rho)$-{\bf {\it covering
property}} if each ball of radius $0<r\leq \rho$ can be covered by
$N$ balls of radius $\frac{r}{\kappa}$.
\end{defn}
Note that when $\rho=\infty$, we simply say that the metric space $(X,d)$ satisfies the $(\kappa,N)$-covering property. It is clear that $(\kappa,N;\rho)$-covering property implies $(\kappa,N;\lambda)$-covering property for any $0<\lambda\leq\rho$. 
\begin{lemma}{\rm (\cite[Corollary 2.3]{CM} and \cite[Lemma 2.1]{CEG})}\label{CM} Let $(X,d,\mu)$ be an $m-m$ space satisfying the $(4,N;\rho)$-covering
property. For every $n\in\mathbb{N}^*$, let  $0<r\leq\rho$
be such that for each $x\in X$,
$\mu(B(x,r))\leq\frac{\mu(X)}{4{N}^2n}$. Then there exists a family
$\mathcal{A}=\{(A_i,A_i^r)\}_{i=1}^{n}$ of capacitors in $X$ such that
 \begin{enumerate}[(a)]
\item for each $i$, $\mu(A_i)\geq \frac{\mu(X)}{2Nn}$, and
\item the subsets ${\{A_i^r\}}_{i=1}^{n}$ are mutually disjoint.
\end{enumerate}
\end{lemma}
We define the \textit{dilatation} of a function $f:(X,d)\to\R$ as
\[{\rm dil}(f)=\sup_{x\neq y}\frac{|f(x)-f(y)|}{d(x,y)},\]
and the \textit{local dilatation} at $x\in X$ as
\[{\rm dil}_x(f)={\rm lim}_{\varepsilon\to0}{\rm dil}(f|_{_{B(x,\varepsilon)}}).\]
When
different distance functions are considered, ${\rm dil}_d(f)$ and
${\rm dil}_{d,x}(f)$ stand for the dilatation and local dilatation
at $x$ associated with the {distance}
$d$ respectively.
A map $f$ is called Lipschitz if ${\rm dil}(f)<\infty$. Let {$(M,g)$} be
a Riemannian manifold and {$d_g$} be  {the distance} associated to the
Riemannian metric $g$. A Lipschitz function on a Riemannian
manifold {$M$} is differentiable almost everywhere and
$|\nabla_gf(x)|$ coincides with ${\rm dil }_x(f)$ almost
everywhere.
Hence, $|\nabla_gf(x)|\leq{\rm dil}(f)$ almost everywhere.\\
The following theorem relies on the construction given in the above lemma. It gives a construction of a  family of disjointly supported  functions with  a nice control on  their dilatations.
Before stating the theorem we need to define the following notation. Given a capacitor $(F,G)$, let $ \T(F,G)$ be
the set of all compactly supported real valued functions on $X$ such
that for every $\varphi\in\T(F,G)$ we have
$\rm{supp~}\varphi\subset G^\circ=G\setminus\partial G$ and
$\varphi\equiv1$ in a neighborhood of $F$.
%
\begin{thm}\label{cm3}
 Let positive constants $p,\rho,L$ and $N$ be given and  $(X,d,\mu)$ be an $m-m$ space satisfying the $(4,N;\rho)$-covering
 property and $$\mu(B(x,r))\leq Lr^p,\quad \text{for every}~ x\in X  ~\text{and}~ 0<r\leq\rho.$$
Then for every $n\in \N^*$ and every $r\leq\min\{\rho,
\left(\frac{\mu(X)}{4N^2Ln}\right)^{1/p}\}$  there is a family of
$n$ mutually  disjoint bounded capacitors
$\{(A_i,A_i^r)\}_{i=1}^n$, of
 $X$ and a family
$\{f_i\}$ of $n$ Lipschitz functions with
$f_i\in\mathcal{T}(A_i,A_i^r)$ such that
$\mu(A_i)\geq\frac{\mu(X)}{2Nn}$ and
\begin{equation}\label{3toma}{\rm
dil}_{d}(f_i)\leq\frac{1}{\rho}+(4N^2L)^{1/p}\left(\frac{n}{\mu(X)}\right)^{1/p}.
\end{equation}
\end{thm}
If the condition $\mu(B(x,r))\leq Lr^p$ is satisfied for every
$r>0$ then we take $\rho=\infty$. Hence, the first term on the
right-hand side of the above inequality vanishes.
\begin{proof}[Proof of Theorem \ref{cm3}] According to Lemma \ref{CM}, if the  $m-m$ space
$(X, d, \mu)$ satisfies $(4,N;\rho)$-covering
 property,
 then for every $r\leq\rho$ such that
\begin{equation}\label{key}
        \mu(B(x,r))\leq\frac{\mu(X)}{4N^2n},\quad\forall x\in X,
        \end{equation} we have a family $\{(A_i,A_i^r)\}$ of mutually disjoint capacitors of $X$ with the desired
property mentioned in the theorem. We claim that when
$r\leq\min\{\rho, \left(\frac{\mu(X)}{4N^2Ln}\right)^{1/p}\}$, the
Inequality (\ref{key}) is automatically satisfied. Indeed,
according to the assumptions we have
\[\mu(B(x,r))\leq Lr^p\leq\min\{L\rho^p,\frac{\mu(X)}{4N^2n}\}\leq\frac{\mu(X)}{4N^2n}.\]
We now consider Lipschitz
 functions $f_i$ supported on $A_i^r$ with
$f_i(x)=1-\frac{d(x,A_i)}r$ on $A_i^r\setminus A_i$, $f_i(x)= 1$
on $A_i$ and zero outside of $A_i^r$. One can easily check that
${\rm dil}_d(f_i)\leq\frac{1}{r}$. Hence, we obtain:
\[{\rm dil}_d(f_i)\leq\frac{1}{\rho}+(4N^2L)^{1/p}\left(\frac{n}{\mu(X)}\right)^{1/p}.\]
This completes the proof.
\end{proof}
%
%
Let $(M,g,\mu)$ be a Riemannian manifold endowed with a finite
non-atomic Borel measure $\mu$. We define the following quantity
that coincides with the eigenvalues of the Laplace-Beltrami
operator when $\mu$ coincides with the Riemannian measure $\mu_g$.
\[\lambda_k(M,g,\mu):=\inf_L\sup\{R(f): f\in L\},\]
where $L$ is a $k$-dimensional vector space of Lipschitz functions
and
$$R(f)=\frac{\int_M|\nabla_gf|^2d\mu}{\int_Mf^2d\mu}$$
 The following corollary is a straightforward  consequence of Theorem \ref{cm3} and it is the key result that we use in the next section.
\begin{cor}\label{jadid} Let $(M,g,\mu)$ be a Riemannian manifold with a finite non-atomic Borel measure $\mu$ {and the distance $d_g$ associated to the Riemannian metric $g$}.
If there exists a measure $\nu$ and a distance $d$ so that
\begin{equation}\label{chap310}d(x,y)\leq d_g(x,y),\quad\forall x,y\in
M;\end{equation}
 \begin{equation}\label{chap320}\nu(A)\leq\mu(A)\quad \text{for all measurable subsets}~A ~\text{of}~
 M,\end{equation}
 and  moreover, there exist positive constants $p,\rho, N$ and $L$ so that $(M,d,\nu)$ satisfies the assumptions of Theorem \ref{cm3},
 then, for every $k\in \N^*$ we have
\begin{eqnarray}\label{ch3inq}
 \lambda_k(M,g,\mu)\leq\frac{16N}{\rho^2}\frac{\mu(M)}{\nu(M)}+
16N(8N^2L)^{2/p}\left(\frac{\mu(M)}{\nu(M)}\right)^{1+2/p}\left(\frac{k}{\mu(M)}\right)^{2/p}.
\end{eqnarray}
\end{cor}
\begin{proof} Take $(M,d,\nu)$ as an $m-m$ space. According to Theorem \ref{cm3}, for every $2k\in\N^*$ and every $r\leq\min\{\rho,
\left(\frac{\nu(X)}{4N^2Ln}\right)^{1/p}\}$, we have a family of $2k$ mutually
 disjoint capacitors $\{(A_i,A^r_i)\}_{i=1}^{2k}$
 and $2k$ Lipschitz
functions $f_i$  such that for every $1\leq i\leq 2k$,
$\nu(A_i)\geq\frac{\nu({M})}{4 Nk}$ and the following inequality satisfies almost everywhere.
\begin{equation*}
|\nabla_gf_i|\leq{\rm dil}_{d_g}(f_i)\leq{\rm
dil}_{d}(f_i)\leq\frac{1}{\rho}+(4N^2L)^{1/p}\left(\frac{2k}{\nu({M})}\right)^{1/p},
\end{equation*}
where the last inequality comes form Inequality (\ref{3toma}). Since
 $\mu\geq\nu$, one has
\begin{equation}\label{eq2c}
\mu(A_i)\geq\nu(A_i)\geq\frac{\nu({M})}{4 Nk}.
\end{equation}
 Supports of the $f_i$ are disjoint and
${\sum}_{i=1}^{2k}\mu(A_i^r)\leq \mu({M})$; therefore, at least $k$ of them
have measure smaller than $\frac{\mu({M})}{k}$. Up to re-ordering,
we assume that for the first $k$ of the $A_i^r$,  we have
\begin{equation}\label{eq1c}\mu(A_i^r)\leq\frac{\mu({M})}{k}.\end{equation}
Therefore,
\begin{eqnarray*}
 \lambda_k(M,g,\mu)\leq\max_iR(f_i)&\leq& \max_i\left(\frac{1}{\rho}+(4N^2L)^{1/p}\left(\frac{2k}{\nu({M})}\right)^{1/p}
\right)^2\frac{\mu(A_i^r)}{\mu(A_i)}\\
&\leq&16N\left(\frac{1}{\rho^2}+(4N^2L)^{2/p}\left(\frac{2k}{\nu({M})}\right)^{2/p}\right)\frac{\mu({M})}{\nu({M})}.
\end{eqnarray*}
The last inequality comes from applying Inequalities \ref{eq2c} and \ref{eq1c}, together with using the following inequality. \[(a+b)^2\leq4(a^2+b^2)\quad\forall a,b\in \R.\]
In conclusion, we obtain Inequality (\ref{ch3inq}).
\end{proof}

%
%
\section{Eigenvalues of Immersed Submanifolds of $\R^N$}\label{sec3ch2}
%
 In this section, we prove Theorem \ref{a3}.  Let $S$ be an
$m$-dimensional immersed submanifold of $\R^{m+p}$ (with or
without boundary). {We recall that  $G$ is the Grassmannian  of all
 $m$-vector spaces in $\R^{m+p}$ endowed with the $O(m+p)$- invariant
Haar measure with total measure 1. Let $H$ be an $m$-vector space in $G$ and  $i_H(S):=\sup_{P\bot H}\sharp (S\cap P)$, where $P$ runs over 
affine $p$-planes orthogonal to $H$. We define the \textit{mean intersection index}
of $S$ as follows.
\[
 \bar{\imath}(S):=\int_G i_H(S) dH.
\]
Similarly, for every $r>0$, we define the \textit{ $r$-local
intersection index} of $S$ by:
\[
 \bar{\imath}_r(S):=\sup_{x\in S}\int_G i_H(S\cap B(x,r)) dH,
\]
where $B(x,r)\subset \R^{m+p}$ is an Euclidean ball of radius $r$ centered at $x$.\\
{Let $H\in G$ and $\pi_H:S\to H$ be the
orthogonal projection of $S$ on $H$}.
The following lemma is an extension of   \cite[Lemma 2.1]{CDE}.
%
\begin{lemma}\label{lemma}
Let $S$ be an $m$-dimensional immersed submanifold of $\R^{m+p}$,
(not necessarily without boundary). Then there exists $H_0\in G$ such that the following inequality satisfies
\begin{equation}\label{2first}
 \textrm{Vol}(S)\leq C_m\bar{\imath}(S)\rm{Vol}(\pi_{H_0}(S)),
\end{equation}
where $C_m$ is a constant depending only on  $m$.
\end{lemma}
\begin{proof}
Since for almost all $H\in G$, a  point in $ \pi_H(S)$ has finite
number of  preimages, one can take a generic $H$ and get
\[
\int_{S} \pi_H^* v_H= \int_{S} \vert\theta_H(x)\vert v_{S}\leq
\int_{\pi_H({S})} i_H({S})v_H=i_H({S})\textrm{Vol}(\pi_H({S})),
\]
where $v_{S}$ and $v_H$ are volume elements of $S$ and $H$
respectively and $$\vert\theta_H(x)\vert v_{S}=\pi_H^* v_H.$$
 Now, by integrating over $G$ we get
\begin{eqnarray}\label{nn}
\nonumber\int_{G}i_H({S}) \rm{Vol}(\pi_H({S}))dH & \geq &  \int_{G}dH \int_{S} \vert\theta_H(x)\vert v_{S} \\
\nonumber & = & \int_{S} \left( \int_{G}\vert\theta_H(x)\vert dH\right) v_{S}\\
& = & I(G)\textrm{Vol}({S}),
\end{eqnarray}
where $I(G):=\int_{G}\vert\theta_H(x)\vert dH$. The last equality
comes from the fact that $I(G)$
 does not depend on the
point $x$ (see \cite[page 101]{CDE}). We also have
\begin{eqnarray}
\nonumber \int_{G}i_H({S}) \rm{Vol}(\pi_H(S))dH &\leq&
\sup_H\textrm{Vol}
(\pi_H({S}))\bar{\imath}({S})\label{mm1}\\
\label{mmsec2} &\leq&2\textrm{Vol}(\pi_{H_0}({S})) \bar{\imath}({S}),
\end{eqnarray}
where $H_0$ is an $m$-plane such that
$2\textrm{Vol}(\pi_{H_0}({S}))\geq\sup_H\textrm{Vol}(\pi_H({S}))$.
By Inequalities (\ref{nn}) and {(\ref{mmsec2})}, we get the following
inequality
\[
\textrm{Vol}(\pi_{H_0}({S}))\geq\frac{I(G)\textrm{Vol}({S})}{2\bar{\imath}({S})}.
\]
This proves Inequality (\ref{2first}) with $C_m=\frac{2}{I(G)}$.
 \end{proof}
 Let $M$ be an
 $m$-dimensional immersed submanifold of $\R^{m+p}$.
Throughout the rest of this section,  for every
$\varepsilon\geq0$, $M^D_\varepsilon$ stands for  $M\setminus  D$,
where  $D$ is any open subdomain of $M$ such that $M\setminus D$
{is} a smooth manifold with  smooth boundary and
$\textrm{Vol}(D)\leq\varepsilon\textrm{Vol}(M)$.\\
%
  \begin{cor}\label{3coreu}For all $x\in\R^{m+p}$  and $\varepsilon\geq0$, we have
  \begin{eqnarray}
  \textrm{Vol}\big({M^D_\varepsilon}\cap B(x,s)\big)\leq
 \frac{2\textrm{Vol}(B^m)}{I(G)}\bar{\imath}_r({M^D_\varepsilon})s^m,\quad\forall~0<s\leq
 r\label{an};\\
 \label{3cor}\textrm{Vol}\big({M^D_\varepsilon}\cap B(x,r)\big)\leq \frac{2\textrm{Vol}(B^m)}{I(G)}\bar{\imath}({M^D_\varepsilon})r^m, \quad\forall
 r>0,
  \end{eqnarray}
where $B^m$ is the $m$-dimensional Euclidean unit ball.
\end{cor}
\begin{proof}
Replacing $S$ by ${M^D_\varepsilon}\cap B(x,s)$ in Lemma
{\ref{lemma}}, we obtain
\begin{eqnarray*}\textrm{Vol}\left({M^D_\varepsilon}\cap
B(x,s)\right)&\leq& \frac{2}{I(G)}\bar{\imath}
\left({M^D_\varepsilon}\cap
B(x,s)\right)\textrm{Vol}\left(\pi_{H_0}\left({M^D_\varepsilon}\cap
B(x,s)\right)\right)\\
&\leq&\frac{2\textrm{Vol}(B^m)}{I(G)}\bar{\imath}_s
\left({M^D_\varepsilon}\right) s^m,
\end{eqnarray*}
where $B^m$ is the $m$-dimensional Euclidean unit ball.
The last inequality comes from
$$\textrm{Vol}\left(\pi_{H_0}\left({M^D_\varepsilon}
\cap B(x,s)\right)\right)\leq
\textrm{Vol}\left(\pi_{H_0}\left(B(x,s)\right)\right)\leq\textrm{Vol}(B^m)
s^m$$ Since $\bar{\imath}_s
\left({M^D_\varepsilon}\right)\leq\bar{\imath}_r
\left({M^D_\varepsilon}\right)$ for all $0<s\leq r$ and
$\bar{\imath}_s \left({M^D_\varepsilon}\right)\leq\bar{\imath}
\left({M^D_\varepsilon}\right)$ for all $s>0$, therefore, we
derive Inequalities (\ref{an}) and (\ref{3cor}).
\end{proof}
\begin{rem}\label{rem}
For  $\ep=0$, we have $M^D_\varepsilon=M$. Hence, we have the
Inequalities (\ref{an}) and (\ref{3cor})
 for $M^D_\varepsilon$ replaced by $M$.
\end{rem}
%
%
\begin{proof}[Proof of Theorem \ref{a3}] This theorem is a
straightforward consequence of Corollary \ref{jadid}.  Here, $M$ with the induced metric from $\R^{m+p}$ and the riemannian measure associated to this metric is our metric measure space.  We begin
with giving
 candidates for the {distance} $d$ and the measure $\nu$ appeared in the statement of Corollary \ref{jadid}, such that the assumptions of Corollary \ref{jadid} are
 satisfied. Let $d=d_{_{eu}}$ be the Euclidean distance in $\R^{m+p}$ and
$\nu=\mu_\epsilon^D$ where $\mu_\epsilon^D(A)$ is the Riemannian
volume of $A\cap M^D_\varepsilon$. One can easily check that
$(M,d_{_{eu}})$ has the $(2,N)$-covering property where $N$
depends only on the dimension of the ambient space $\R^{m+p}$.
Moreover, one can consider $N$ as a function depending only on the
dimension $m$ according to the Nash embedding theorem (see \cite[page 106]{CDE}). There also exists $L>0$
such that $\mu_\varepsilon^D(B(x,s))\leq Ls^m$ for $s\leq\rho$.
 We now consider the two following cases:
\begin{itemize} \item Take $\rho=r$. According to Corollary \ref{3coreu}, one can take   $L=
\frac{2\textrm{Vol}(B^m)}{I(G)}\bar{\imath}_r({M^D_\varepsilon})$.
Therefore, Corollary \ref{jadid} implies
\begin{equation}\label{asr3}
\lambda_k(M)\leq\alpha_m\frac{1}{(1-\varepsilon)r^2}+\beta_m\frac{\bar{\imath}_r(M^D_\varepsilon)^{2/m}}{(1-\varepsilon)^{1+2/m}}\left(\frac{k}{{\rm
Vol}(M)} \right)^{2/m}.
\end{equation}
 \item Take $\rho=\infty$. According to Corollary \ref{3coreu}, one can take
$L=\frac{2\textrm{Vol}(B^m)}{I(G)}\bar{\imath}({M^D_\varepsilon}).$
Therefore, Corollary \ref{jadid} implies 
\begin{equation}\label{asr32}
\lambda_k(M)\leq\beta_m\frac{\bar{\imath}(M^D_\varepsilon)^{2/m}}{(1-\varepsilon)^{1+2/m}}\left(\frac{k}{{\rm
Vol}(M)} \right)^{2/m}.
\end{equation}
\end{itemize}
Note that  here we replace  $\nu(M)$ and $\mu(M)$ in  Corollary \ref{jadid} by  $\mu_\epsilon^D(M)$ and ${\rm Vol}(M)$ respectively.
The left hand-sides of Inequalities (\ref{asr3}) and (\ref{asr32})
do not depend on $D$. Hence, taking the infimum  over $D$, we get
Inequalities (\ref{a2}) and (\ref{a1}).
\end{proof}
\section{Eigenvalues of Complex Submanifolds of $\C P^N$}\label{sec4ch2}
In this section, we provide the proof of  Theorem \ref{01}. Before
going into the proof we need to recall the universal inequality
proved by El Soufi, Harrell and
 Ilias which is the key idea of the proof. The following lemma is  a special case of that universal inequality \cite[Theorem 3.1]{EHI} (see also \cite{ChY}):
\begin{lemma}\label{3univ}
Let $M^m$ be a compact complex manifold of complex dimension $m$
and $\phi:M\to\C P^N$ be a holomorphic immersion. Then the
eigenvalues of the Laplace-Beltrami operator on
$(M,\phi^*g_{_{FS}})$ satisfy the following inequality:
\begin{equation}\label{b}
\sum_{i=1}^{k}(\lambda_{k+1}-\lambda_i)^2\leq\frac{2}{m}\sum_{i=1}^{k}(\lambda_{k+1}-\lambda_i)(\lambda_i+c_m),
\end{equation}
where $c_m=2m(m+1)$.
\end{lemma}
Another useful result is the following recursion formula given by
Cheng and Yang:
\begin{lemma}(\cite[Corollary 2.1]{CY}) If  a positive sequence of numbers
$\mu_1\leq\mu_2\leq\cdots\leq\mu_{k+1}$,  satisfies the following
inequality
 \begin{equation}\label{3cheng0}
\sum_{i=1}^{k}(\mu_{k+1}-\mu_i)^2\leq\frac{4}{n}\sum_{i=1}^{k}\mu_i(\mu_{k+1}-\mu_i),
\end{equation}
then
\begin{equation*}\mu_{k+1}\leq(1+\frac{4}{n})k^{2/n}\mu_1.
\end{equation*}
\end{lemma}
\begin{thm}\label{ch301}
Let $M^m$ be a compact complex manifold of complex dimension $m$
admitting a holomorphic immersion $\phi:M\to\C P^N$. Then for
every
 $k\in\N^*$ we have
\begin{equation}\label{3}\lambda_{k+1}(M,\phi^*g_{_{FS}})\leq
2(m+1)(m+2)k^{\frac{1}{m}}-2m(m+1).
\end{equation}
\end{thm}
\begin{proof}[Proof of Theorem \ref{ch301}]According to Lemma \ref{3univ}, the eigenvalues of the Laplace operator on
$M$ satisfy universal Inequality (\ref{b}). We replace $\lambda_i$
by $\mu_i:=\lambda_i+c_m$ in Inequality (\ref{b}) and we obtain:
\begin{equation*}
\sum_{i=1}^{k}(\mu_{k+1}-\mu_i)^2\leq\frac{2}{m}\sum_{i=1}^{k}\mu_i(\mu_{k+1}-\mu_i).
\end{equation*}  One now has a positive
sequence of numbers $\mu_1\leq\mu_2\leq\cdots\leq\mu_{k+1}$ that
satisfies Inequality (\ref{3cheng0})  with $n=2m$. Applying the
recursion formula of Cheng and Yang, we get the following
inequality:
\begin{equation}\label{4}\mu_{k+1}\leq(1+\frac{4}{2m})k^{2/2m}\mu_1.
\end{equation} By replacing $\mu_i$ by $\lambda_i+c_m$ in
Inequality (\ref{4}), we obtain:
\[\lambda_{k+1}(M,\phi^*g_{_{FS}})\leq(1+\frac{2}{m})(\lambda_1(M,\phi^*g_{_{FS}})+c_m)k^{1/m}-c_m.\]
Since $M$ is a compact manifold,
$\lambda_1(M,\phi^*(g_{_{FS}})=0$. Therefore,
\[\lambda_{k+1}(M,\phi^*g_{_{FS}})\leq(1+\frac{2}{m})c_mk^{1/m}-c_m=2(m+1)(m+2)k^{1/m}-2m(m+1),\]
which completes the proof.
\end{proof}
As we mentioned in the introduction, for $k=1$ we get a sharp
upper bound:
\begin{equation}\label{3first}\lambda_2(M,\phi^*g_{_{FS}})\leq\lambda_2(\C
P^m,g_{_{FS}})=4(m+1).\end{equation} In \cite{BLY}, Bourguignon,
Li and Yau obtained  an upper bound for the first non-zero
 eigenvalue of a complex manifold $(M,\omega)$ which admits a \textit{full}\label{chp3full} holomorphic immersion (i.e. $\Phi(M)$ is not  contained in any hyperplane of $\C P^N$)  into $\C P^N$.
\begin{equation}\label{bly}
\lambda_2(M,\omega)\leq4m\frac{N+1}{N}d([\Phi],[\omega]).
\end{equation}
Here, $d([\Phi],[\omega])$ is the \textit{holomorphic immersion
degree} -- a homological invariant -- defined by
\[d([\Phi],[\omega])=\frac{\int_M\Phi^*(\omega_{_{FS}})\wedge\omega^{m-1}}{\int_M\omega^m},\]
where $\omega_{_{FS}}$ is the K\"ahler form of $\C P^N$ with
respect to
the Fubini-Study metric and $\omega$ is K\"ahler form on $M$.\\
If one takes $\omega=\Phi^*(\omega_{_{FS}})$, then
$d([\Phi],[\omega])=1$ and we get Inequality (\ref{3first}) as a
corollary of Inequality (\ref{bly}). Theorem \ref{ch301} gives 
another proof of this sharp inequality. Moreover, it gives upper bounds for higher eigenvalues of complex submanifolds of $\C P^N$ endowed with the Fubini-Study metric.

\end{document}